\documentclass[12pt]{amsart}
\usepackage{amssymb}
\usepackage{t1enc}
\usepackage[mathscr]{eucal}
\usepackage{parskip}
\usepackage{xcolor}
\usepackage{tikz}
\usepackage{tikz-cd}
\usepackage{hyperref}
\usepackage[capitalize,noabbrev,nameinlink]{cleveref}
\numberwithin{equation}{section}
\usepackage[margin=2.9cm]{geometry}
\usepackage{booktabs}
\usepackage{enumerate}
\usepackage{bm}
\usepackage{xfrac}
\usepackage[citestyle=alphabetic,bibstyle=alphabetic]{biblatex}
\bibliography{toda}

\usepackage{listings}
\theoremstyle{plain}
\newtheorem{thm}{Theorem}[section]
\newtheorem{lem}[thm]{Lemma}
\newtheorem{cor}[thm]{Corollary}
\newtheorem{prop}[thm]{Proposition}
\newtheorem{conj}[thm]{Conjecture}

 \theoremstyle{definition}
\newtheorem{defn}[thm]{Definition}
\newtheorem{rem}[thm]{Remark}

\newtheorem{notn}[thm]{Notation}
\newtheorem{ques}[thm]{Question}

\newcommand{\mb}[1]{\mathbb{#1}}

\newcommand{\vphi}{\varphi}

\makeatletter
\@namedef{subjclassname@2020}{\textup{2020} Mathematics Subject Classification}
\makeatother

\begin{document}
%%%%%%%%%%%%%%%%%%%%%%%%%%%%%%%%%%%%%%%%
\title{Toda primes}

\author{Stephen McKean}
\address{Department of Mathematics \\ Brigham Young University} 
\email{mckean@math.byu.edu}
\urladdr{shmckean.github.io}

\subjclass[2020]{Primary: 11A41. Secondary: 11B68, 55Q40.}
%%%%%%%%%%%%%%%%%%%%%%%%%%%%%%%%%%%%%%%%

\begin{abstract}
A Toda prime of an integer $n$ is an odd prime $p$ such that $4n=(p-1)k$ with $k$ coprime to $p$. We conjecture that every positive integer admits at least two Toda primes. We give a partial proof that every positive integer admits at least one Toda prime. We conclude by discussing connections to denominators of Bernoulli numbers and a generalization of Sophie Germain primes.
\end{abstract}

\maketitle

\section{Introduction}
The fourth stable homotopy group of spheres is trivial, meaning that $\pi_{n+4}(S^n)=0$ for all $n>5$. In contrast to this, it is a theorem that $S^m$ has no trivial higher homotopy groups when $m\in\{2,3,4,5\}$, as we will briefly explain.

Curtis proved that $\pi_n(S^4)\neq 0$ for all $n\geq 4$ \cite{Cur69}. Curtis also proved that $\pi_n(S^2)\neq 0$ for all $n\not\equiv 1\mod 8$. These same results were obtained (via different methods) by Mimura, Mori, and Oda \cite{MMO75}. The proof that $\pi_n(S^5)\neq 0$ for all $n\geq 5$ was given by Mori \cite{Mor75} and Mahowald \cite{Mah75,Mah82}.

Since $\pi_n(S^2)\cong\pi_n(S^3)$ for all $n\geq 3$, the remaining case was $\pi_n(S^3)$ with $n\equiv 1\mod 8$. This last case was proved by Gray \cite{Gra84}, and later by Ivanov, Mikhailov, and Wu \cite{IMW16} using different methods. In \emph{op.~cit.}, the authors note that \cite[Theorem 5.2(ii)]{Tod66} implies that
\[\mb{Z}/p\subseteq\pi_{2(p-1)k+1}(S^3)\]
whenever $\gcd(p,k)=1$ \cite[p.~342, Equation (B)]{IMW16}. It follows that if every positive integer $n$ admits an odd prime $p$ and an integer $k$ such that $\gcd(p,k)=1$ and $4n=(p-1)k$, then $\pi_n(S^3)\neq 0$ for all $n\equiv 1\mod 8$. This leads one to the following definition.

\begin{defn}
    Let $n$ be an integer. A \emph{Toda prime} of $n$ is an odd prime $p$ such that $p-1\mid 4n$ and $\gcd(p,\frac{4n}{p-1})=1$. Denote the set of Toda primes of $n$ by $T(n)$ (see \cref{table:toda sets}), and let $t(n):=|T(n)|$ (see \cref{fig:tn}).
\end{defn}

\begin{table}[t]
\centering
\caption{Sets of Toda primes}\label{table:toda sets}
\begin{tabular}{|r|l|}
\hline
$n$ & $T(n)$\\
\hline
1 & $3,5$\\
2 & $3,5$\\
3 & $5,7,13$\\
4 & $3,5,17$\\
5 & $3,11$\\
6 & $5,7,13$\\
7 & $3,5,29$\\
8 & $3,5,17$\\
9 & $5,7,13,19,37$\\
10 & $3,11,41$\\
\hline
\end{tabular}
\quad
\begin{tabular}{|r|l|}
\hline
$n$ & $T(n)$\\
\hline
11 & $3,5,23$\\
12 & $5,7,13,17$\\
13 & $3,5,53$\\
14 & $3,5,29$\\
15 & $7,11,13,31,61$\\
16 & $3,5,17$\\
17 & $3,5$\\
18 & $5,7,13,19,37,73$\\
19 & $3,5$\\
20 & $3,11,17,41$\\
\hline
\end{tabular}
\quad
\begin{tabular}{|r|l|}
\hline
$n$ & $T(n)$\\
\hline
21 & $5,13,29,43$\\
22 & $3,5,23,89$\\
23 & $3,5,47$\\
24 & $5,7,13,17,97$\\
25 & $3,11,101$\\
26 & $3,5,53$\\
27 & $5,7,13,19,37,109$\\
28 & $3,5,17,29,113$\\
29 & $3,5,59$\\
30 & $7,11,13,31,41,61$\\
\hline
\end{tabular}
\end{table}

If every positive integer has a Toda prime, then one can greatly simplify the proof that $\pi_n(S^3)\neq 0$ for $n\geq 3$. This was asked on MathOverflow \cite{MO} (and attributed to Roman Mikhailov) several years ago.

\begin{ques}[Mikhailov]\label{ques:at least one}
    Does every positive integer have a Toda prime?
\end{ques}

In fact, it appears that every positive integer has at least two Toda primes.

\begin{conj}\label{conj:at least two}
    If $n$ is a positive integer, then $t(n)\geq 2$.
\end{conj}

We tried to answer \cref{ques:at least one} in the affirmative, but our approach hits a snag. To turn our failed attempt into a theorem, we adopt the time-tested tradition of stating our snag as a conjecture (\cref{conj:main}). We will give some heuristic evidence for this conjecture in \cref{sec:heuristic}.

\begin{conj}\label{conj:main}
    Let $n$ be an odd, square-free multiple of 3. Assume that there exists $p\in\{5,7,13\}$ such that $p\mid n$, and that $r\nmid n$ for $r\in\{5,7,13\}-\{p\}$. Finally, assume there exists $q\in T(3p)-\{5,7,13\}$ such that $q\nmid n$. Then $t(n)\geq 4$.
\end{conj}

\begin{thm}\label{thm:main}
    Assume \cref{conj:main}. If $n$ is a positive integer, then $t(n)\geq 1$. If $5\nmid n$, then $t(n)\geq 2$. If $3\mid n$, then $t(n)\geq 3$.
\end{thm}

In \cref{sec:lemmas}, we state and prove a few simple lemmas. We prove \cref{thm:main} in \cref{sec:proof} by inducting on the number of odd prime factors of $n$. Essentially all of the real work happens in \cref{lem:t(n)=3}. We conclude with \cref{sec:observations}, where we pose a couple questions that arose while working on this project.

\subsection*{Acknowledgements}
We thank Nick Andersen, Pace Nielsen, and Kyle Pratt for helpful conversations. The author was partially supported by the NSF (DMS-2502365) and the Simons Foundation.

\section{Some lemmas}\label{sec:lemmas}
There are many elementary statements that one can prove about Toda primes. In this section, we record a few lemmas and corollaries that we will need en route to \cref{thm:main}.

\begin{notn}
    Given a positive integer $n$, let $\Omega(n)$ denote the set of prime divisors of $n$. Later, we will use the notation $\omega(n):=|\Omega(n)|$.
\end{notn}

\begin{lem}\label{lem:prime factors not toda}
    If $n$ is a positive integer, then $T(n)\cap\Omega(n)=\varnothing$.
\end{lem}
\begin{proof}
    Let $p$ be an odd prime factor of $n$. If $p-1$ divides $4n$, then we have $4n=(p-1)k$ for some $k$. In particular, $p\mid(p-1)k$, so $p\mid k$ by Euclid's lemma. It follows that $\gcd(p,k)\neq 1$, so $p\not\in T(n)$.
\end{proof}

\begin{lem}\label{lem:subset}
    Let $a,n$ be positive integers. Then
    \[T(an)\supseteq T(a)\cup T(n)-\Omega(an).\]
\end{lem}
\begin{proof}
    Suppose that $p\in T(a)\cup T(n)-\Omega(an)$. Since $p\in T(a)$, we have $4a=(p-1)k$ for some $k$ coprime to $p$. Thus $4an=(p-1)kn$, and since $p\not\in\Omega(an)$, we have $kn$ coprime to $p$ by Euclid's lemma. Thus $p\in T(an)$.
\end{proof}

\begin{cor}\label{cor:multiply by prime}
    If $p$ is prime, then $T(pn)\supseteq T(n)-\{p\}$.
\end{cor}
\begin{proof}
    Note that $\Omega(pn)=\Omega(n)\cup\{p\}$. We have $T(n)\cap\Omega(n)=\varnothing$ by \cref{lem:prime factors not toda}, so $T(n)-\Omega(pn)=T(n)-\{p\}$. The claim thus follows from \cref{lem:subset}.
\end{proof}

\begin{cor}\label{cor:multiply by divisor}
    If $a\mid n$, then $T(an)\supseteq T(n)$.
\end{cor}
\begin{proof}
    Since $a\mid n$, we have $\Omega(an)=\Omega(n)$. We have seen that $T(n)\cap\Omega(n)=\varnothing$ (\cref{lem:prime factors not toda}), so $T(an)\supseteq T(n)-\Omega(n)=T(n)$ by \cref{lem:subset}.
\end{proof}

\begin{cor}\label{cor:odd}
    If $n$ is a positive integer, then $T(2n)\supseteq T(n)$.
\end{cor}
\begin{proof}
    Toda primes are odd by definition, so this follows from \cref{cor:multiply by prime}.
\end{proof}

\section{Hunting for a Toda prime}\label{sec:proof}
To begin our proof of \cref{thm:main}, we will use the following lemma to reduce to the case of $3\mid n$.

\begin{lem}\label{lem:not 5 or not 3 and 5}
    Let $p\in\{3,5\}$. If $p\nmid n$, then $p\in T(n)$. In particular:
    \begin{itemize}
        \item If $5\nmid n$, then $t(n)\geq 1$.
        \item If $3,5\nmid n$, then $t(n)\geq 2$.
    \end{itemize}
\end{lem}
\begin{proof}
    Let $p\in\{3,5\}$. Then $p-1\mid 4n$, and $\gcd(4n,p)=1$ by assumption. Thus $\gcd(\frac{4n}{p-1},p)=1$, so $p\in T(n)$.
\end{proof}

Next, we characterize the Toda primes of certain multiples of 3 by inducting on $\omega(n)$.

\begin{lem}\label{lem:t(n)=3}
    Assume \cref{conj:main}. Let $n$ be an odd, square-free multiple of 3. Then $t(n)\geq 3$. Moreover, $t(n)=3$ if and only if $T(n)=\{5,7,13\}$.
\end{lem}
\begin{proof}
    We will induct on $\omega(n)$. Our base cases will consist of $\omega(n)\leq 4$. Note that if $p\nmid n$ for each $p\in\{5,7,13\}$, then $T(n)\supseteq\{5,7,13\}$. In particular, we may restrict our attention to multiples of these three primes. Moreover, $t(ap)\geq t(a)$ for any prime $p\not\in T(a)$ by \cref{cor:multiply by prime}, so we may assume that every prime factor $p\mid n$ is a Toda prime of some divisor of $n$.
    \begin{itemize}
        \item The case of $\omega(n)=1$ is just the calculation $T(3)=\{5,7,13\}$.
        \item For $\omega(n)=2$, we just need to compute $t(15)=t(39)=5$ and $t(21)=4$.
        \item For $\omega(n)=3$, we first compute $t(3\cdot 5\cdot 7)=9$ and $t(3\cdot 5\cdot 13)=t(3\cdot 7\cdot 13)=8$. It remains to compute, for each $p\in\{5,7,13\}$, the Toda primes of $3pq$ for each $q\in T(3p)$. Using the code provided in \cref{sec:code}, we find that $t(3pq)\geq 4$ for all such $p,q$.
        \item For $\omega(n)=4$, we first compute $t(3\cdot 5\cdot 7\cdot 13)=16$. For the remaining computations in this case, we use the code in \cref{sec:code}.
        \begin{itemize}
            \item If $\{p,q\}\subseteq\{5,7,13\}$ and $r\in T(3pq)$, then $t(3pqr)\geq 9$.
            \item If $p\in\{5,7,13\}$ and $\{q,r\}\subseteq T(3p)$, then $t(3pqr)\geq 7$.
            \item If $p\in\{5,7,13\}$, $q\in T(3p)$, and $r\in T(3pq)$, then $t(3pqr)\geq 5$.
        \end{itemize}
    \end{itemize}
    
    Now let $m\geq 4$. Assume that if $a$ is an odd, square-free multiple of 3 such that $\omega(a)\leq m$, then $t(a)\geq 3$, with $t(a)=3$ if and only if $T(a)=\{5,7,13\}$.
    
    Let $n$ be an odd, square-free multiple of 3 satisfying $\omega(n)=m+1$. Then there are distinct primes $p_1,\ldots,p_m>3$ such that
    \[n=3\cdot\prod_{i=1}^m p_i.\]
    For each $1\leq j\leq m$, let
    \[n_j:=3\cdot\prod_{i\neq j}p_i.\]
    We first prove that $t(n)\geq 3$. By our inductive hypothesis, we have $t(n_j)\geq 3$ for all $j$. If $t(n_j)>3$ for some $j$, then $t(n)\geq t(n_j)-1>2$ by \cref{cor:multiply by prime}. We may thus assume that $t(n_j)=3$ for all $j$. Similarly, if $p_j\not\in T(n_j)$ for some $j$, then $t(n)\geq t(n_j)\geq 3$, so we may assume that $p_j\in T(n_j)$ for all $j$. But since $t(n_j)=3$ for all $j$, we have $T(n_j)=\{5,7,13\}$ by the inductive hypothesis. Thus $p_j\in\{5,7,13\}$ for all $1\leq j\leq 4$, contradicting our assumption that $p_i\neq p_j$ for $i\neq j$. It follows that $t(n)\geq 3$.
    
    It remains to show that $t(n)=3$ if and only if $T(n)=\{5,7,13\}$. Assume that $t(n)=3$. Suppose that $T(n)\neq\{5,7,13\}$. By our inductive hypothesis, we have $t(n_j)\geq 3$ for all $j$, with $t(n_j)=3$ if and only if $T(n_j)=\{5,7,13\}$. \cref{cor:multiply by prime} states that $T(n)\supseteq T(n_j)-\{p_j\}$. Thus if $t(n_j)=3$ and $p_j\not\in\{5,7,3\}$, then $T(n)=\{5,7,13\}$, as desired. We may therefore assume that if $t(n_j)=3$, then $p_j\in\{5,7,13\}$. We thus have at most three $j$ such that $t(n_j)=3$.

    Note that if $n$ is coprime to each of 5, 7, and 13, then $T(n)\supseteq\{5,7,13\}$, so our assumption that $t(n)=3$ implies that $T(n)=\{5,7,13\}$. We may thus assume that there is at least one $j$ such that $p_j\in\{5,7,13\}$. Moreover, if there exist two such factors, say $p_i,p_j\in\{5,7,13\}$, then we have $p_i\in T(n_j)$ and $p_j\in T(n_i)$. But $p_i\mid n_j$ and $p_j\mid n_i$, so this would contradict \cref{lem:prime factors not toda}. Thus $n$ is divisible by exactly one of 5, 7, and 13. We may thus assume $p_1\in\{5,7,13\}$ and $p_i\not\in\{5,7,13\}$ for $i>1$.

    Now $T(n_1)=\{5,7,13\}$, and $T(n)\supseteq\{5,7,13\}-\{p_1\}$. Since $t(n)=3$, there is some prime $q\in T(n)-\{5,7,13\}$. Applying \cref{cor:multiply by prime} again, we have that $t(n)=3\geq t(n_j)-1$, so $t(n_j)\leq 4$ for all $j$. As $T(n)\supseteq T(n_j)-\{p_j\}$, we find that if $t(n_j)=4$, then $T(n_j)=T(n)\cup\{p_j\}$. Thus $q\in T(n_j)$ for all $j>1$, and hence $T(n_j)=\{5,7,13,q,p_j\}-\{p_1\}$.

    By definition, $q-1\mid 4n_j$ for all $j>1$, so $q-1$ divides
    \[\gcd(4n_2,\ldots,4n_m)=12p_1.\]
    That is, $q\in T(3p_1)-\{5,7,13\}$. Together with our previous observation that $p_1\in\{5,7,13\}$ and $p_i\not\in\{5,7,13\}$ for $i>2$, \cref{conj:main} now implies that $t(n)\geq 4$. This contradicts our assumption that $t(n)=3$, and hence we find that $T(n)=\{5,7,13\}$.
    
    The converse consists of counting to three, and we are done.
\end{proof}

As a corollary, we (conditionally) find that if $3\mid n$, then $t(n)\geq 3$.

\begin{cor}\label{prop:divisible by 3}
    Assume \cref{conj:main}. If $3\mid n$, then $t(n)\geq 3$.
\end{cor}
\begin{proof}
    By \cref{cor:multiply by divisor}, we may assume that $n$ is square-free. By \cref{cor:odd}, we may further assume that $n$ is odd. The result now follows from \cref{lem:t(n)=3}.
\end{proof}

Combining \cref{lem:not 5 or not 3 and 5} and \cref{prop:divisible by 3} gives us \cref{thm:main}. In order to go from $t(n)\geq 1$ to $t(n)\geq 2$ (assuming \cref{conj:main}), we would need to prove that $t(5a)\geq 2$ for any odd, square-free $a$ that is coprime to 15. This seems much harder than \cref{prop:divisible by 3}. The qualitative difference between these two cases is that there are multiples of 5 with distinct Toda sets of size 2, such as $T(5)=\{3,11\}$ and $T(55)=\{3,23\}$, while $T(n)=\{5,7,13\}$ whenever $3\mid n$ and $t(n)=3$ (assuming \cref{conj:main}).

\section{Heuristic for \cref{conj:main}}\label{sec:heuristic}
The criteria in \cref{conj:main} imply that $\{5,7,13,q\}-\{p\}\subseteq T(n)$. In particular, $t(n)\geq 3$ for any such $n$. In this section, we will explain why we expect $t(n)\geq 4$ for such $n$.

Firstly, if there exists $q'\in T(3p)-\{5,7,13,q\}$ such that $q'\nmid n$, then $\{5,7,13,q,q'\}-\{p\}\subseteq T(n)$, and we are done. In fact, the Toda primes of $n$ are precisely those primes among $\{2d+1:d\mid 2n\}$ that are not factors of $n$. Thus if
\begin{equation}\label{eq:heuristic}
\{2d+1\text{ prime}:d\mid 2n\}-(\Omega(n)\cup\{5,7,13,q\})
\end{equation}
is non-empty, then $t(n)\geq 4$. Our heuristic for \cref{conj:main} is that the set $\{2d+1:d\mid 2n\}$ consists of $2^{\omega(n)+1}$ elements, while $\Omega(n)\cup\{5,7,13,q\}$ consists of $\omega(n)+3$ elements. 

The discrepancy between the growth of $\{2d+1:d\mid 2n\}$ and $\omega(n)$ holds for any $n$, not just those satisfying the assumptions of \cref{conj:main}. However, computations suggest that there is no bound on $\omega(n)$ among integers with $t(n)\leq 3$ when no further assumptions are placed on $n$. The assumptions of \cref{conj:main} imply that any counterexample would need to be divisible by essentially all Toda primes of its divisors, which provides an upper bound on the size of $n$ for a given $\omega(n)$ and hence a lower bound on the density of primes in the relevant interval. 

Roughly, a non-optimal bound for such $n$ with $\omega(n)=m+1$ is given by setting $n=3p_1\cdots p_m$, with $p_1\in\{5,7,13\}$ and $p_{i+1}=12p_1\cdots p_i+1$ for each $1\leq i\leq m$. One can then calculate $p_{i+1}=12^ip_1+\sum_{j=0}^{i-1}12^j$, and hence $n=3\cdot 12^{\frac{m^2-m}{2}}p_1^m+o(p_1^{m-1})$. The density of primes less than $4n+1$ is bounded below by $\log(4n+1)^{-1}$, which (ignoring lower order terms) is about
\[\left(\frac{m^2-m+2}{2}\log(12)+m\log(p_1)\right)^{-1}.\]
Now, it should be noted that the author is not an analyst, so the above may be rife with error. But it appears that the growth of the number of candidate Toda primes, i.e.~the set $\{2d+1:d\mid 2n\}$, outstrips the growth of $\Omega(n)$ and the dropping density of primes less than the largest Toda candidate, and so we expect \cref{eq:heuristic} to be non-empty.

To conclude this section, we present some computational evidence for \cref{conj:main} in \cref{table:conjecture}. Among the set of $n$ satisfying the criteria of \cref{conj:main}, we calculate the minimal $t(n)$ for a given $\omega(n)$. To introduce notation, let
\[\Upsilon(n)=\min\{t(a):\omega(a)=n\text{ and }a\text{ satisfies the assumptions of \cref{conj:main}}\}.\]
Since $t(ap)\geq t(a)$ for any prime $p\not\in T(a)$ (\cref{cor:multiply by prime}), we may restrict our attention to $n$ such that every prime factor $p\mid n$ is a Toda prime of some divisor of $n$. Some rough code for this computation can be found in \cref{sec:code}. 

\begin{rem}
    In the range we have computed, we have $\Upsilon(n)=a(n+4)-1$, where $a(n)$ is given by \cite[\href{https://oeis.org/A118096}{A118096}]{oeis}. However, this coincidence should not continue, as A118096 has terms satisfying $a(n+1)<a(n)-1$, whereas we expect $\Upsilon(n+1)\geq\Upsilon(n)-1$ for all $n$ in light of \cref{cor:multiply by prime}. 
\end{rem}

\begin{table}[t]
\centering
\caption{Minimal $t(a)$ for \cref{conj:main}}\label{table:conjecture}
\begin{tabular}{|c|c|}
\hline
$n$ & $\Upsilon(n)$\\
\hline
2 & 4\\
3 & 4\\
4 & 5\\
5 & 7\\
6 & 7\\
7 & 7\\
8 & 11\\
\hline
\end{tabular}
\end{table}

\section{Some observations}\label{sec:observations}
It took a fair amount of floundering to come up with the induction argument used in \cref{lem:t(n)=3}. Along the way, we noticed a few things that seem worth mentioning.

\subsection{Bernoulli denominators}
There is a close connection between Bernoulli denominators and Toda primes. We will introduce this relationship by characterizing the Toda primes of primes greater than 5.

\begin{prop}\label{prop:t(p)}
Assume $p\geq 7$ is a prime. Let $\vphi$ denote the totient function. If $\vphi(x)=4p$ for some integer $x$, then $T(p)=\{3,5,2p+1\}$ or $\{3,5,4p+1\}$. Otherwise, $T(p)=\{3,5\}$.
\end{prop}
\begin{proof}
    One can directly check that $3,5\in T(p)$ for all primes greater than 5. Now by Euler's product formula, we have $\vphi(x)=p_1^{e_1-1}(p_1-1)\cdots p_m^{e_m-1}(p_m-1)$, where $x=\prod_{i=1}^m p_i^{e_i}$ is the prime factorization of $x$. It follows that there exists $x$ such that $\vphi(x)=4p$ if and only if one of the following cases holds:
    \begin{enumerate}[(i)]
    \item $x=2^2\cdot q$, where $q$ is an odd prime such that $q-1=2p$. In this case, $q$ is a Toda prime of $p$ with $\frac{4p}{q-1}=2$.
    \item $x=2^r\cdot 3\cdot q$, where $r\in\{0,1\}$ and $q$ is an odd prime such that $q-1=2p$. In this case, $q$ is a Toda prime of $p$ with $\frac{4p}{q-1}=2$.
    \item $x=2^r\cdot q$, where $r\in\{0,1\}$ and $q$ is an odd prime such that $q-1=4p$. In this case, $q$ is a Toda prime of $p$ with $\frac{4p}{q-1}=1$.
    \item $x=2^r\cdot 5^2$, where $r\in\{0,1\}$ (in which case $p=5$). This case is not relevant for this lemma, as we have assumed $p\geq 7$.
    \end{enumerate}
    It remains to show that no other primes can be the Toda prime of $p$. To this end, let $q>5$ be a Toda prime of $p$. Then $q-1\mid 4p$, so we either have $q-1=4p$ or $q-1=2p$ (as $q-1$ is even and $p$ is odd). The existence of such a $q$ gives us a solution to $\vphi(x)=4p$ as outlined in cases (i), (ii), and (iii).

    To conclude, we need to show that we cannot have two odd primes $q_1,q_2$ such that $q_1-1=2p$ and $q_2-1=4p$. In other words, we need to show that there is no prime $p\geq 7$ such that both $2p+1$ and $4p+1$ are prime. It in fact suffices to assume $p>3$ here: under this assumption, $p\not\equiv 0\mod 3$, which implies that either $2p+1$ or $4p+1$ is divisible by 3. In particular, at most two of $p,2p+1,4p+1$ can be prime when $p>3$.
\end{proof}

\begin{rem}
    The sequence of primes (greater than 5) with $t(p)=2$ and $t(p)=3$ can be found at \cite[\href{https://oeis.org/A043297}{A043297} and \href{https:\\oeis.org/A087634}{A087634}]{oeis}, respectively. In particular, for all primes other than $5$, we have $t(p)=2$ if and only if the denominator of the Bernoulli number $B_{4p}$ is 30, which we now prove.
\end{rem}

\begin{prop}\label{prop:bernoulli 30}
    Let $D_{2n}$ denote the denominator of the Bernoulli number $B_{2n}$. For all primes other than 5, we have $t(p)=2$ if and only if $D_{4p}=30$.
\end{prop}
\begin{proof}
    We can check the claim directly for $p=2$ and $p=3$, so we may assume $p>5$. To begin, assume $t(p)=2$. Since $p>5$, we have $T(p)=\{3,5\}$ by \cref{prop:t(p)}. By the von Staudt--Clausen theorem, $D_{4p}$ is the product of all primes $q$ such that $q-1\mid 4p$. Clearly $q=2,3,5$ all satisfy this condition, so $30\mid D_{4p}$. The only other candidates for $q$ are $q=2p+1$ and $q=4p+1$. But $\gcd(2p+1,2)=\gcd(4p+1,1)=1$, so either of these options would be a Toda prime of $p$ if they were prime. Since $2p+1,4p+1\not\in T(p)$, it follows that $2p+1$ and $4p+1$ are not prime, and thus $D_{4p}=30$.

    Now assume that $D_{4p}=30$. Then the only primes $q$ such that $(q-1)\mid 4p$ are 2, 3, and 5, so $t(p)\leq 2$. It follows that $t(p)=2$, as desired.
\end{proof}

Other results analogous to \cref{prop:bernoulli 30} are more or less straightforward to prove on a case-by-case basis. For example:

\begin{prop}\label{prop:bernoulli}
    The following statements are true:
    \begin{enumerate}[(i)]
    \item If $D_{12m}=2730$, then $T(3m)=\{5,7,13\}$.
    \item If $D_{20m}=330$, then $T(5m)=\{3,11\}$.
    \item If $D_{60m}=56786730$, then $T(15m)=\{7,11,13,31,61\}$.
    \end{enumerate}
\end{prop}

We have omitted the proof of \cref{prop:bernoulli} for the sake of brevity. The strategy is to show that if $p$ is an odd prime factor of the Bernoulli denominator but not a Toda prime of $n$, then there are additional shifted primes dividing $4n$. This leads us to the following conjecture.

\begin{notn}\label{notn:bernoulli}
    Given an even integer $2m$, let $P(2m)$ denote the set of primes such that $p-1\mid 2m$, so that $D_{2m}=\prod_{p\in P(2m)}p$. Let $F(d)=\min\{2m>0:D_{2m}=d\}$ \cite{PW23}.
\end{notn}

\begin{conj}\label{conj:general bernoulli}
    Let $d$ be a Bernoulli denominator. 
    \begin{enumerate}[(i)]
    \item\label{conj:toda for bernoulli} If $F(d)=4a$ for some integer $a$, then $T(am)=T(a)$ whenever $D_{4am}=D_{4a}$. 
    \item\label{conj:bernoulli inequality} If $D_{4am}\neq D_{4a}=d$, then $t(am)\geq t(a)$.
    \end{enumerate}
\end{conj}

The inequality in \cref{conj:general bernoulli} \eqref{conj:bernoulli inequality}, if true, is tight. For example, $T(55)=\{3,23\}$ (so $t(55)=t(5)$), while $D_{220}=7590\neq 330=D_{20}$. Note that our proof of \cref{prop:divisible by 3} conditionally proves \cref{conj:general bernoulli} for $a=3$. Proving \cref{conj:general bernoulli} \eqref{conj:bernoulli inequality} for $a=3$ and $a=5$ would resolve \cref{conj:at least two}. 

\subsection{Germane primes and two wavefronts}
\cref{prop:bernoulli 30} and \cref{prop:bernoulli} are special cases of \cref{conj:general bernoulli} \eqref{conj:toda for bernoulli}. It is clear that the same procedure should yield a proof of \cref{conj:general bernoulli} \eqref{conj:toda for bernoulli} for any chosen Bernoulli denominator, but it is unclear how to bootstrap the proof to the general case. Our best guess is as follows. One would need to show that if $p\mid m$ for $p\in T(a)-\{3,5\}$, then there exists an odd prime $q\not\in T(a)$ such that $q-1\mid 4am$. To prove \cref{prop:bernoulli}, one can show that there exists an odd prime $q\not\in T(a)\cup\{3,5\}$ such that $q-1\mid 4pa$. It therefore suffices to show that at least one element of $\{2pi+1:i\mid 2a\}$ is prime. We summarize this approach in the following lemma.

\begin{lem}\label{lem:strategy for denoms}
    Let $d$ be a Bernoulli denominator with $F(d)=4a$ for some integer $a$. If $\{2pi+1:i\mid 2a\}$ contains a prime number for each $p\in T(a)$, then \cref{conj:general bernoulli} \eqref{conj:toda for bernoulli} holds for this Bernoulli denominator.
\end{lem}
\begin{proof}
    We know that $p-1\mid 4a$ with $\gcd(\frac{4a}{p-1},p)=1$ for all $p\in T(a)$. Thus $p-1\mid 4am$, and we have $\gcd(\frac{4am}{p-1},p)=1$ if and only if $p\nmid m$. It therefore suffices to show that if $p\mid m$ for some $p\in T(a)$, then $D_{4am}>D_{4a}$.
    
    Suppose $p\mid m$. If $p\in T(a)$, then $p\nmid a$, and hence any prime of the form $q=2pi+1$ (with $i\mid 2a$) cannot satisfy $q-1\mid 4a$. In particular, $q\not\in T(a)$. So if $\{2pi+1:i\mid 2a\}$ contains a prime number $q$, then $q-1\mid 4am$ and $q-1\nmid 4a$. Thus $D_{4am}>D_{4a}$, as desired. 
\end{proof}

A na\"ive guess is that $\{2pi+1:i\mid 2a\}$ contains a prime of the form $p(q-1)+1$, where $q\in\Omega(d)$. This led us to plot the first 1000 odd primes against the proportion of primes among $\{p(q-1)+1\}$, where $q$ ranges among the first 10000 odd primes (see \cref{fig:ratios}). There appear to be two families or wavefronts of primes in this plot, which we cannot explain. In this range, the Sophie Germain primes all belong to the lower wavefront.

\begin{defn}
    A prime is called \emph{germane} if it is of the form $p(q-1)+1$, where $p$ and $q$ are both prime. We will call $p$ and $q$ the \emph{width} and \emph{length}, respectively, of $p(q-1)+1$. For example, every Sophie Germain prime is the width of a length 3 germane prime, while 3 is the only germane prime of length 2. We will also say that a germane prime of width $p$ is \emph{germane to $p$}.
\end{defn}

\begin{ques}
    Given an odd prime $p$, let $r_p(n)$ denote the ratio of width $p$ germane primes among $\{p(q-1)+1\}$, where $q$ ranges among the first $n$ odd primes. Does the double wavefront pattern in \cref{fig:ratios} persist in the limit
    \[\lim_{n\to\infty}\{(p,r_p(n)):p\text{ prime}\}?\]
    Is there a qualitative description (beyond the frequency of primes germane to $p$) of the primes falling into each of these two families?
\end{ques}

\begin{ques}
    In \cref{fig:germane}, we plot all germane primes of width $p$ (horizontal axis) and length $q$ (vertical axis), where $p$ and $q$ range among the first 1000 primes. There are vertical lines indicating primes that are the width of very few germane primes, and horizontal lines indicating primes that are rarely the length of a germane prime. Can any of these lines be explained?
\end{ques}

\begin{ques}
Given a prime $r$, let $w(r)=\#\{p\text{ prime}:r\text{ germane to }p\}$. Given an integer $n\geq 0$, what is the density of the level set
\[\{r\text{ prime}:w(r)=n\}?\]
It appears that the drop-off of these densities is quite stark, with primes satisfying $w(r)=6$ apparently only occurring every 80000 primes or so.
\end{ques}

\appendix
\section{Code}\label{sec:code}
Here is a basic Sage program for computing $T(n)$ and $t(n)$, along with computations of $t(n)$ for $\omega(n)=3$ and $\omega(n)=4$ when $n$ is divisible by $3p$ for $p\in\{5,7,13\}$.

\begin{lstlisting}
# Computing the set of Toda primes
def T(n):
    S = [];
    for i in divisors(2*n):
        d = 2*i;
        if is_prime(d+1) and gcd(d+1,4*n/d) == 1:
            S.append(d+1);
    return(S)

# omega(n) = 3 case
print('p','q','t(3pq)')
for p in [5,7,13]:
    for q in T(3*p):
        print(p,q,len(T(3*p*q)))

# omega(n) = 4 case
# p,q in {5,7,13}, r in T(3pq)
print('p','q','r','t(3pqr)')
for Q in Combinations([5,7,13],2):
    for r in T(3*prod(Q)):
        n = 3*prod(Q)*r;
        print(Q[0],Q[1],r,len(T(n)))
        
# p in {5,7,13}, q,r in T(3p)        
print('p','q','r','t(3pqr)')
for p in [5,7,13]:
    for Q in Combinations(T(3*p),2):
        n = 3*p*prod(Q);
        print(p,Q[0],Q[1],len(T(n)))

# p in {5,7,13}, q in T(3p), r in T(3pq)
print('p','q','r','t(3pqr)')
for p in [5,7,13]:
    for q in T(3*p):
        for r in T(3*p*q):
            n = 3*p*q*r;
            print(p,q,r,len(T(n)))
\end{lstlisting}

Here is code for calculating the minimal $t(n)$ over all $n$ satisfying the assumptions of \cref{conj:main} and with a given $\omega(n)$. We used ChatGPT-5.1 to define the recursive functions \verb|select_toda_primes| and \verb|select_recursive|, so these could certainly be improved. 

The search space of $n$ grows very rapidly as $\omega(n)$ increases. It took our machine about 10 hours to calculate \verb|minimal(8)|. One could significantly cull the search space for \verb|minimal(omega)| by taking those $n$ among the search space for \verb|minimal(omega-k)| satisfying \verb|len(T(n)) < minimal(omega-k)+k| and then selecting $k$ more prime factors from among the relevant sets of Toda primes, since each additional prime factor can decrease $t(n)$ by at most 1.

\begin{lstlisting}
# Computing the set of Toda primes
def T(n):
    S = [];
    for i in divisors(2*n):
        d = 2*i;
        if is_prime(d+1) and gcd(d+1,4*n/d) == 1:
            S.append(d+1);
    return(S)

def select_toda_primes(m,n):
    results = []
    select_recursive(m, n, [], results)
    return(results)

def select_recursive(m_needed,current_n,collected,\
results):
    # If done, record the result
    if m_needed == 0:
        results.append(collected.copy())
        return

    current_T = T(current_n)

    # Take k primes from T(current_n)
    for k in range(1,min(len(current_T), m_needed)+1):

        # Choose all combinations of size k
        for combo in Combinations(current_T, k):

            new_collected = collected + list(combo)

            # Multiply n by product of chosen primes
            prod = 1
            for p in combo:
                prod *= p
            next_n = current_n * prod

            # Recurse for remaining primes
            select_recursive(m_needed-k, next_n,\
            new_collected, results)
    
# Check if list of primes satisfies criteria
# of conjecture
def satisfies_criteria(n,p):
    basic_primes = list(set([5,7,13])-set([p]))
    todas = list(set(T(3*p))-set([5,7,13]))
    if all([n % q != 0 for q in basic_primes])\
    and any([n % q != 0 for q in todas]):
        return(True)
    else: return(False)

# Calculate min t(n)
def minimal(omega):
    min = 10^10
    for p in [5,7,13]:
        for primes in select_toda_primes(omega-2,3*p):
            n = 3*p*prod(primes)
            if satisfies_criteria(n,p)\
            and len(T(n)) < min:
                min = len(T(n))
    return(min)
\end{lstlisting}

\printbibliography

\begin{figure}[p]
    \includegraphics[width=.75\linewidth]{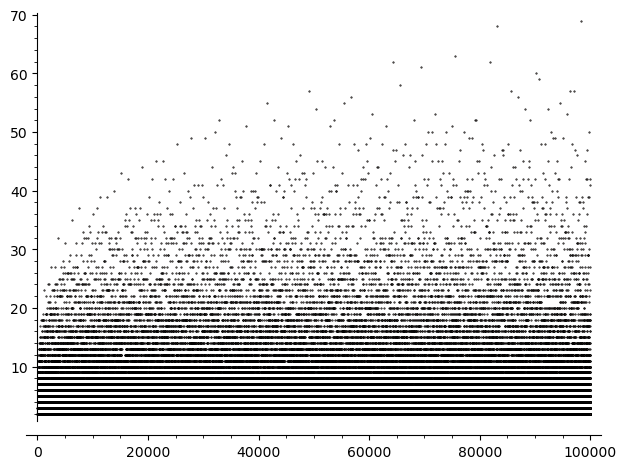}
    \caption{$t(n)$ for $n\leq 100000$}\label{fig:tn}
\end{figure}

\begin{figure}[p]
    \includegraphics[width=.75\linewidth]{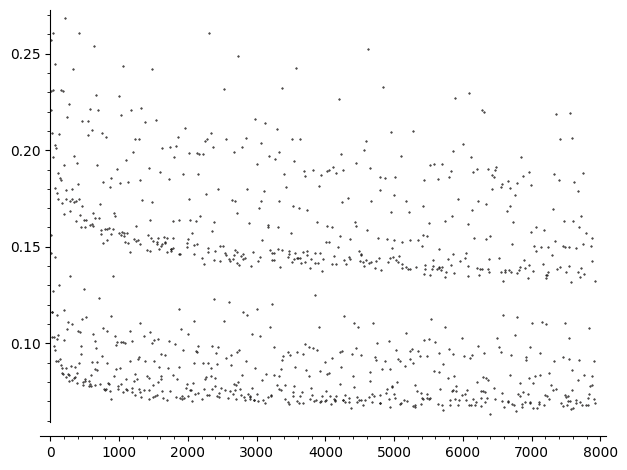}
    \caption{Ratio of primes germane to $p$}\label{fig:ratios}
\end{figure}

\begin{figure}[p]
    \includegraphics[width=.75\linewidth]{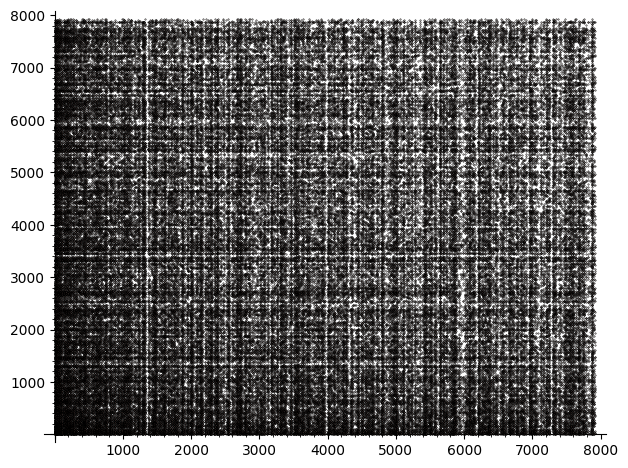}
    \caption{Width versus length of germane primes}\label{fig:germane}
\end{figure}

\end{document}